\documentclass[11pt]{amsart}
\usepackage{setspace}
\usepackage{palatino, mathpazo}
\usepackage[margin=2.5cm]{geometry}
\usepackage[OT2,T1]{fontenc}
\usepackage[utf8]{inputenc}
\usepackage{tikz-cd}
\usepackage{imakeidx}
 \usepackage{comment}
\makeindex
\usepackage{microtype,hyperref,amsfonts,amssymb,amsthm,mathrsfs,mathabx}
\usepackage{graphicx}
\usepackage[
backend=biber,
style=alphabetic-verb,
sorting=nyt,
maxalphanames=9,
doi=false,url=false,isbn=false,
maxbibnames=9
]{biblatex}
\usepackage{stmaryrd}
\usepackage[colorinlistoftodos]{todonotes}
\usepackage{amssymb}
\usepackage{amsmath}
\usepackage{latexsym}
\usepackage[all]{xy}
\usepackage{fancyhdr}
\usepackage{calligra}
\usepackage{indentfirst}
\usepackage{graphicx}
\usepackage{esint}
\usepackage{newlfont}
\usepackage{amscd}
\usepackage{amsmath}
\usepackage{latexsym}
\usepackage{tikz-cd}
\usepackage{enumitem}
\theoremstyle{plain}
\numberwithin{equation}{section}
\theoremstyle{theorem}
\newtheorem{thm}[equation]{Theorem}
\newtheorem{prop}[equation]{Proposition}

\newtheorem{cor}[equation]{Corollary}
\newtheorem{lem}[equation]{Lemma}

\theoremstyle{definition}

\theoremstyle{remark}
\newtheorem{rmk}[equation]{Remark}

\newcommand{\E}{\mathcal{E}}

\renewcommand{\C}{\mathbb C}

\newcommand{\R}{\mathbb R}

\newcommand{\N}{\mathbb N}

\renewcommand{\P}{\mathbb P}

\DeclareMathOperator{\Rep}{Re}

\newcommand{\OO}{\mathcal O}

\renewcommand{\epsilon}{\varepsilon}

\newcommand{\db}{\bar{\partial}}

\newcommand{\tr}{\mathrm{tr}}

\DeclareMathOperator\supp{supp}

\DeclareMathOperator{\vol}{vol}

\DeclareMathOperator{\PSH}{PSH}

\DeclareMathOperator{\Hilb}{Hilb}

\DeclareMathOperator{\sHom}{\mathscr{H}\text{\kern -3pt {\calligra\large om}}\,}
\newcommand{\p}{\partial}

\title{Quantization for Semipositive Adjoint Line Bundles}
\author{Yu-Chi Hou}

\begin{document}
\begin{abstract}
Let \(L\) be a big and semipositive line bundle on a complex projective manifold \(X\), and let \(\theta\in c_1(L)\) be a smooth semipositive representative. In the adjoint setting \(H^0(X,L^k\otimes K_X)\), we prove that Donaldson's quantized Monge--Amp\`ere energy converges to the Monge--Amp\`ere energy for every bounded \(\theta\)-plurisubharmonic function. This extends the quantization picture from the ample case to the big and semipositive setting, where smooth positive representatives are no longer available and non-pluripolar Monge--Amp\`ere theory is required. The main new input is a comparison theorem between adjoint Bergman kernels and their small ample twists. As a consequence, we prove that the normalized adjoint Bergman measures converge weakly to the corresponding non-pluripolar Monge--Amp\`ere measures. Our result partially answers a question of Berman--Freixas i Montplet concerning the convergence of quantized Monge--Amp\`ere energies in the semipositive setting.
\end{abstract}

\keywords{K\"ahler quantization, Monge--Amp\`ere energy, Bergman kernel, non-pluripolar Monge--Amp\`ere measure, semipositive line bundle, weak geodesic}
	\maketitle
   \section{Introduction}
Let \(L\) be an ample line bundle over a complex projective manifold \(X\), and let \(\omega\) be the K\"ahler form polarized by \(L\). A central theme in K\"ahler geometry, first articulated by Yau \cite[p.~139]{Yau}, is the approximation of transcendental geometric objects on \((X,\omega)\) by asymptotic algebraic data arising from the tensor powers \(L^k:=L^{\otimes k}\). This is the philosophy of \emph{K\"ahler quantization}.

The starting point is the approximation of the infinite-dimensional space
\[
\mathcal{H}_\omega:=\{\phi\in C^\infty(X):\omega+dd^c\phi>0\}
\]
of K\"ahler potentials by the finite-dimensional spaces \(\mathcal{H}_k\) of positive Hermitian forms on \(H^0(X,L^k)\). This was established in the seminal works of Tian \cite{Tian}, Bouche \cite{Bouche}, Ruan \cite{Ruan}, Catlin \cite{Catlin}, Zelditch \cite{Zelditch}, and Lu \cite{Lu}, among many others, via asymptotic expansions of the Bergman kernel.

Donaldson later proposed that not only the metrics but also the \emph{geometry} of \(\mathcal{H}_\omega\) should be captured by the finite-dimensional geometries of \(\mathcal{H}_k\) \cite[p.~483]{DonaldsonI}. This principle has since been realized in a variety of settings and has had a significant impact on problems concerning canonical metrics and stability; see, for example, \cite{PS,SongZelditchqu,CS,BerndtssonProb,DLR,DonaldsonI,DonaldsonScalII,RTZ,Zhang} and the references therein. We refer to \cite{Mama} for an exposition of some classical developments in this direction.

A natural question is whether this quantization picture extends beyond the ample case. This question has become increasingly relevant in view of recent progress on canonical metrics in degenerate or singular settings \cite{BBJ,LTW,Li,DZ24,Dervan,Xu,PT25-1,PTT23,PT25-2}, where the positivity of \(L\) is available only in a weaker sense. The present paper fits into this broader direction by establishing a quantization result for the Monge--Amp\`ere energy in the big and semipositive setting.

\subsection*{Statement of Main Results}
We now describe our main results, referring to Section~\ref{backgrounds on pluripotential theory} and Section~\ref{Bergman Measures, Positivity of Direct Images, and Equilibrium} for details. Let \(L\) be a holomorphic line bundle over a complex projective manifold \(X\) of dimension \(n\). We assume that \(L\) is big and is endowed with a smooth Hermitian metric \(h_0\) whose Chern curvature form $
\theta:=c_1(L,h_0)$ is semipositive.

One of the main difficulties in the degenerate setting is the lack of smooth potentials. Instead, we work with the set of \(\theta\)-plurisubharmonic functions on \(X\), denoted by \(\PSH(X,\theta)\). This space parametrizes singular Hermitian metrics on \(L\) whose curvature \(\theta_u:=\theta+dd^c u\) is positive in the sense of currents. Following the pioneering work of Boucksom et al.\ \cite{BEGZ}, for \(u\in \PSH(X,\theta)\) one can define the \textit{non-pluripolar Monge--Amp\`ere measure} \(\theta_u^n\), which is a Borel measure on \(X\) charging no pluripolar sets. We denote by
\[
E_\theta:\PSH(X,\theta)\to [-\infty,\infty)
\]
the \textit{Monge--Amp\`ere energy}, or \textit{Aubin--Yau energy}. It is a primitive of the Monge--Amp\`ere operator and plays an important role in the variational approach to canonical K\"ahler metrics.

To define the quantized analogue of \(E_\theta\), we consider the adjoint bundle \(L^k\otimes K_X\), where \(K_X:=\Omega_X^{n,0}\) denotes the canonical bundle of \(X\). Fix a smooth volume form \(\nu\) on \(X\). Viewing \(\nu\) as a smooth Hermitian metric on \(K_X^*\), it induces a smooth Hermitian metric \(\nu^{-1}\) on \(K_X\) as follows: for each \(x\in X\) and \(\xi\in (K_X)_x=\Omega_{X,x}^{n,0}\), we set
\[
|\xi|^2_{\nu^{-1}(x)} := \frac{i^{n^2}\,\xi\wedge \overline{\xi}}{\nu}(x).
\]

Given a function \(u:X\to[-\infty,\infty)\), we define the \textit{Hilbert norm} on \(H^0(X,L^k\otimes K_X)\) by the \(L^2\)-norm of sections of the adjoint bundle:
\begin{equation}\label{Hilbert norm}
  \Hilb_k(u)(s,s)=\|s\|^2_{\Hilb_k(u)}:=\int_X (h_0^k\otimes \nu^{-1})(s,s)\,e^{-ku}\,d\nu,
\end{equation}
whenever the integral is finite for the given section \(s\in H^0(X,L^k\otimes K_X)\). As we will see later (cf.\ Lemma~\ref{Hilbert norm independent of the choice of volume form}), \(\Hilb_k\) is in fact independent of the choice of \(\nu\).

We say that \(u\) is \textit{admissible} if \(\Hilb_k(u)\) is finite on the whole space \(H^0(X,L^k\otimes K_X)\) for every \(k\in\N\). Clearly, every bounded function is admissible. For an admissible function \(u\), we define the \textit{(adjoint) quantized Monge--Amp\`ere energy} by
\begin{equation}\label{quantized MA-intro}
    E_k(u):=-\frac{1}{kN_k}\log\left(\frac{\det \Hilb_k(u)}{\det \Hilb_k(0)}\right),
\end{equation}
which was introduced by Donaldson \cite{DonaldsonScalII}; see also Berman--Boucksom \cite{BB10}.

Our main result is the following quantization statement for the Monge--Amp\`ere energy.

\begin{thm}\label{quantization of MA energy}
If \(L\) is a big and semipositive line bundle on \(X\), then for any \(u\in \PSH(X,\theta)\cap L^\infty\) we have
\[
\lim_{k\to\infty}E_k(u)=E_\theta(u).
\]
\end{thm}

In \cite[Remark~3.6]{BF14}, Berman--Freixas i Montplet raised the question of whether Donaldson's quantized Monge--Amp\`ere energy converges to the Monge--Amp\`ere energy in the semipositive setting for finite-energy \(\theta\)-plurisubharmonic functions. Theorem~\ref{quantization of MA energy} establishes this convergence for bounded \(\theta\)-plurisubharmonic functions and thus provides a partial answer.

When \(L\) is ample, Donaldson \cite{DonaldsonScalII} first established the convergence for smooth K\"ahler potentials in the untwisted case, that is, without \(K_X\)-twisting. In the adjoint setting, Berman--Freixas i Montplet \cite[Theorem~3.5]{BF14} proved the quantization of energy for finite-energy potentials, and Darvas--Xia \cite[Theorem~2.11]{DX22} subsequently obtained the corresponding statement for arbitrary twisting. In the big untwisted case, Berman--Boucksom \cite[Theorem~A]{BB10} proved convergence for all continuous metrics. More recently, Darvas--Xia \cite{DX24} proved a \emph{partial quantization} theorem for continuous metrics with arbitrary twisting on a pseudoeffective line bundle.

Theorem~\ref{quantization of MA energy} has an immediate application to the quantization of non-pluripolar Monge--Amp\`ere measures. For an admissible weight \(u\), let \(\{S_1,\dots,S_{N_k}\}\) be an \(\Hilb_k(u)\)-orthonormal basis of \(H^0(X,L^k\otimes K_X)\), where \(N_k=\dim_\C H^0(X,L^k\otimes K_X)\). We define the \textit{Bergman kernel} and \textit{Bergman measure} by
\begin{equation}\label{Bergman measure}
K^k_{u,\nu}:=\sum_{j=1}^{N_k}(h_0^k\otimes\nu^{-1})(S_j,S_j),
\qquad
B^k_{u}:=K^k_{u,\nu}e^{-ku}\,d\nu.
\end{equation}
The measure \(B_u^k\) is independent of the choice of \(\nu\) (see Lemma~\ref{Hilbert norm independent of the choice of volume form}), although \(K^k_{u,\nu}\) depends on \(\nu\).

Combining Theorem~\ref{quantization of MA energy} with a standard trick of Berman--Boucksom--Witt Nystr\"om \cite{BBWN}, we obtain the following weak convergence of normalized Bergman measures to non-pluripolar measures.

\begin{cor}\label{convergence of Bergman measures}
Let \(u\in \PSH(X,\theta)\cap L^\infty(X)\). Then \(N_k^{-1}B_u^k\) converges weakly to
\(\vol(L)^{-1}\theta_u^n\).
\end{cor}
Our approach refines the arguments of \cite{BF14} in a way adapted to the semipositive setting. The main new input is a comparison theorem between the adjoint Bergman kernels of \(L^k\otimes K_X\) and the small ample twists \(L^k\otimes K_X\otimes A^{\lfloor \epsilon k\rfloor}\), with an almost optimal asymptotic bound (Theorem~\ref{Bergman kernel comparison}). This allows us to reduce the problem to an \(\epsilon\)-twisted ample situation, where Berman's local Morse inequality \cite[Proposition~2.5]{BF14} applies with explicit control in \(\epsilon\); see \eqref{p-twisted-local-holomorphic-morse-eps}. Combining this with the standard variational relations linking the Monge--Amp\`ere energy to non-pluripolar Monge--Amp\`ere measures and the quantized energy to Bergman measures, together with the convexity of the quantized Monge--Amp\`ere energy along weak geodesics, we prove Theorem~\ref{quantization of MA energy} for bounded potentials. As explained in Remark~\ref{pitfall for finite energy potentials}, the argument does not currently extend to arbitrary finite-energy potentials because of integrability issues for the initial tangent vectors of weak geodesics.

As this work was being finalized, Siarhei Finski kindly informed us of his related work \cite{FinskiBig}, in which he studies the quantization of the \(d_p\)-geometry on big line bundles for potentials of the form \(P_\theta(f)\) (see \S\ref{MA envelope}), where \(f\) is continuous. In particular, when \(p=1\) and \(\theta\) is semipositive, his results yield a convergence statement for the quantized energy for this class of potentials, which forms a subclass of \(\PSH(X,\theta)\cap L^\infty\). Our setting and approach differ in two main respects. First, his quantization is formulated using the untwisted spaces \(H^0(X,L^k)\), whereas we work in the adjoint setting \(H^0(X,L^k\otimes K_X)\), which allows us to apply the Berndtsson--P\u{a}un positivity theorem directly. Second, he reduces to the ample case by passing to suitable birational models over \(X\), whereas our reduction proceeds by twisting \(L^k\otimes K_X\) by an ample line bundle on \(X\), as explained above.
\subsection*{Organizations}
We briefly describe the structure of the paper. 

In Section~2 we collect the necessary background from pluripotential theory, following \cite{BEGZ,BBGZ,DDL18,DDL18L1}. We also recall the basic variational formulas for \(E_\theta\).

Section~3 is devoted to Bergman measures and quantized Monge--Amp\`ere energies in the adjoint setting. We establish basic properties of \(B_u^k\) and \(E_k(u)\), discuss their variational formulas, and adapt the arguments of \cite{Ber09,BB10} to prove the convergence of Bergman measures at equilibrium in the adjoint setting. 

Section~4 contains the main new input. We prove the comparison theorem (Theorem~\ref{Bergman kernel comparison}), based on the construction of peak sections with asymptotically optimal \(L^\infty\)-control (cf.\ Lemma~\ref{semi-classical Ohsawa--Takegoshi}), and derive the corresponding local Morse inequality in the ample twisting setting. 

Finally, in Section~5 we combine these ingredients to prove Theorem~\ref{quantization of MA energy} and Corollary~\ref{convergence of Bergman measures}. We also explain in Remark~\ref{pitfall for finite energy potentials} why the argument does not currently extend to arbitrary finite-energy potentials.
\section*{Acknowledgments}
This work was supported by NSF Grant DMS-2405274 and by the Hauptman Fellowship from the Department of Mathematics, University of Maryland. The author is grateful to his advisor, Tam\'as Darvas, for his support and guidance. The author also thanks Siarhei Finski for enlightening lectures and discussions at the summer school held at the R\'enyi Institute in Budapest, for generously sharing his work \cite{FinskiBig} and related insights, and for pointing out a mistake in an earlier version.
\section{Background on Pluripotential Theory}\label{backgrounds on pluripotential theory}
Let \(X\) be a compact K\"ahler manifold of dimension \(n\) with K\"ahler form \(\omega\). Let \(L\) be a holomorphic line bundle over \(X\), and let \(h_0\) be a smooth Hermitian metric on \(L\). Locally, \(h_0\) is given by a weight \(e^{-\Phi_0}\), that is, for a holomorphic frame \(e_L\) over an open set \(U\subset X\),
\[
h_0(e_L,e_L)=e^{-\Phi_0},\qquad \Phi_0\in C^\infty(U,\R).
\]
The Chern curvature \(\theta:=c_1(L,h_0)\) is a closed real \((1,1)\)-form on \(X\), locally given by
\[
\theta|_U=dd^c\Phi_0,
\]
where we adopt the convention
\[
d^c:=\frac{i}{4\pi}(\db-\p),\qquad \text{so that}\qquad dd^c=\frac{i}{2\pi}\p\db.
\]
By Chern--Weil theory, the cohomology class \(c_1(L):=\{\theta\}\in H^{1,1}(X,\R)\) is independent of the choice of \(h_0\) and is called the \textit{first Chern class} of \(L\).

A function \(u:X\to \R\cup\{-\infty\}\) is called \textit{\(\theta\)-plurisubharmonic} (or simply \(\theta\)-psh) if \(u+\Phi_0\) is plurisubharmonic on \(U\). We denote by \(\PSH(X,\theta)\) the set of \(\theta\)-psh functions. Thus, \(\PSH(X,\theta)\) parametrizes singular Hermitian metrics \(h_0e^{-u}\) on \(L\) with curvature current
\[
\theta_u:=\theta+dd^cu\geq 0 \quad \text{in the sense of currents}.
\]

We say that a holomorphic line bundle \(L\) over \(X\) is
\begin{itemize}
\item[(i)] \textit{positive} if there exists \(\phi\in C^\infty(X,\R)\) such that \(\theta_\phi\) is a K\"ahler form;
\item[(ii)] \textit{semipositive} if there exists \(\phi\in C^\infty(X,\R)\) such that \(\theta_\phi\) is a semipositive \((1,1)\)-form;
\item[(iii)] \textit{big} if there exists \(u\in \PSH(X,\theta)\) such that \(\theta_u\geq \epsilon\omega\) for some \(\epsilon>0\).
\end{itemize}

From now on, we work in the following set-up:
\begin{equation}\label{set-up}
\begin{aligned}
&\text{Let \(L\) be a big line bundle admitting a smooth Hermitian metric \(h_0\)}\\
&\text{with semipositive curvature \(\theta=c_1(L,h_0)\geq 0\).}
\end{aligned}
\end{equation}
A result of Ji--Shiffman \cite{JiShiffman} shows that a compact K\"ahler manifold \(X\) admitting a big line bundle is automatically projective.

For \(1\leq p\leq n\) and \(u_1,\dots,u_p\in \PSH(X,\theta)\), we recall the non-pluripolar Monge--Amp\`ere product introduced in \cite{BEGZ,GZ07}. Using Bedford--Taylor theory \cite{BT76}, consider the canonical truncations
\[
u_k^{(j)}:=\max(u_k,-j),\qquad j\in\N,
\]
and the sequence of positive currents
\[
\mathbf{1}_{\bigcap_{k=1}^p\{u_k>-j\}}\,
\theta_{u_1^{(j)}}\wedge \cdots \wedge \theta_{u_p^{(j)}}.
\]
It was shown in \cite[\S 1]{BEGZ} that this sequence is increasing and converges weakly to a positive closed \((p,p)\)-current
\[
\theta_{u_1}\wedge \cdots \wedge \theta_{u_p},
\]
called the \textit{non-pluripolar Monge--Amp\`ere product} of \(u_1,\dots,u_p\). In particular, when \(p=n\) and \(u=u_1=\cdots=u_n\), this yields a non-pluripolar Borel measure \(\theta_u^n\), called the \textit{non-pluripolar Monge--Amp\`ere measure} of \(u\). All complex Monge--Amp\`ere measures in this paper are understood in this sense.

We define the \textit{volume} of \(L\) by
\begin{equation}\label{volume of a (1,1)-class}
\vol(L):=\int_X\theta^n.
\end{equation}
By the work of Boucksom \cite{BoucksomVolume}, since \(L\) is big, we have \(\vol(L)>0\), and the above analytic definition coincides with the following algebraic definition:
\begin{equation}\label{volume of a line bundle}
\vol(L)=\limsup_{k\to\infty}\frac{h^0(X,L^k)}{k^n/n!},
\end{equation}
where \(h^0(X,L^k):=\dim_\C H^0(X,L^k)\). Moreover, this is in fact a limit by a theorem of Fujita \cite[\S 11.4]{LazarsfeldII}. For any holomorphic line bundle \(F\) over a projective manifold \(X\), the argument in \cite[Example~1.2.33]{LazarsfeldI} shows that the volume is stable under twisting:
\begin{equation}\label{stability of volume}
\lim_{k\to\infty}\frac{h^0(X,L^k\otimes F)}{k^n/n!}=\vol(L).
\end{equation}

For \(f\in C^\infty(X)\), we define the \textit{Monge--Amp\`ere envelope} of \(f\) by
\begin{equation}\label{MA envelope}
P_\theta(f):=\sup\{u\in \PSH(X,\theta):u\leq f\},
\end{equation}
and we denote the \textit{contact set} by
\[
D_\theta(f):=\{x\in X:P_\theta(f)=f\}.
\]
Berman \cite[Theorem~1.1]{Ber09} showed that the \textit{equilibrium measure} satisfies
\begin{equation}\label{Concentration of equilibrium measure}
\theta_{P_\theta(f)}^n=\mathbf{1}_{D_\theta(f)}\,\theta_f^n,
\end{equation}
where \(\theta_f^n\) is the ordinary wedge product. He also proved \cite[Proposition~3.1(iii)]{Ber09} that
\begin{equation}\label{contact set is contained in semipositive set}
D_\theta(f)\subset \{\theta_f\geq 0\}.
\end{equation}
Since \(\theta\geq 0\), the constant \(c:=\inf_X f\) is admissible in \eqref{MA envelope}, and hence \(P_\theta(f)\in \PSH(X,\theta)\cap L^\infty(X)\), with
\[
c\leq P_\theta(f)\leq C:=\sup_X f.
\]

For \(u\in \PSH(X,\theta)\cap L^\infty(X)\), we define the Monge--Amp\`ere energy by
\begin{equation}\label{MA energy}
E_\theta(u):=\frac{1}{(n+1)\vol(L)}\sum_{j=0}^n\int_X u\, \theta_u^j\wedge \theta^{n-j}.
\end{equation}
Clearly, \(E_\theta(u)\leq E_\theta(v)\) if \(u\leq v\). Following \cite[Definition~2.9]{BEGZ}, we extend \(E_\theta\) to \(\PSH(X,\theta)\) by
\[
E_\theta(u):=\inf\{E_\theta(v): v\in \PSH(X,\theta)\cap L^\infty(X),\ u\leq v\},\qquad u\in \PSH(X,\theta),
\]
and define the \textit{finite energy space} by
\begin{equation}\label{E1 space}
\E^1(X,\theta):=\{u\in \PSH(X,\theta):E_\theta(u)>-\infty\}.
\end{equation}
We summarize some properties of \(E_\theta\) and \(\E^1(X,\theta)\); see \cite{BEGZ,DDL18L1,DDL18}.

\begin{thm}\label{Properties on Energy and E1}
In the set-up \eqref{set-up},
\begin{itemize}
\item[(i)] If \(u\in \E^1(X,\theta)\), then \(\theta_u^n\) has full Monge--Amp\`ere mass and \(u\in L^1(X,\theta_u^n)\), i.e.,
\[
\int_X\theta_u^n=\vol(L),\qquad \int_X|u|\,\theta_u^n<\infty.
\]
\item[(ii)] \(E_\theta\) is continuous along decreasing sequences.
\item[(iii)] \(E_\theta\) is concave along affine curves. In particular, for \(u,v\in \E^1(X,\theta)\),
\begin{equation}\label{Energy estimate}
\frac{1}{\vol(L)}\int_X (u-v)\,\theta_u^n\leq E_\theta(u)-E_\theta(v)\leq \frac{1}{\vol(L)}\int_X (u-v)\,\theta_v^n.
\end{equation}
Consequently, if \(v\leq u\), then \(E_\theta(v)\leq E_\theta(u)\). This also implies that if \(u\in \PSH(X,\theta)\), \(v\in \E^1(X,\theta)\), and \(v\leq u\), then \(u\in \E^1(X,\theta)\).
\item[(iv)] If \(u,v\in \E^1(X,\theta)\), then \(P_\theta(u,v)\in \E^1(X,\theta)\).
\end{itemize}
\end{thm}

We now briefly recall the construction of weak (Mabuchi) geodesics, following Berndtsson \cite{Berndtsson15} in the K\"ahler case and its extension to the big setting in \cite{DDL18,DDL18L1}. Given a curve \([0,1]\ni t\mapsto v_t\in \PSH(X,\theta)\), we consider its complexification on the strip
\[
T:=\{\tau\in \C:0<\Rep\tau<1\},
\]
defined by
\[
X\times T\ni (x,\tau)\longmapsto V(x,\tau):=v_{\Rep\tau}(x),
\]
where \(\pi_X:X\times T\to X\) is the projection. We say that \(t\mapsto v_t\) is a \textit{subgeodesic} if \(V\in \PSH(X\times T,\pi_X^*\theta)\). For \(u_0,u_1\in \PSH(X,\theta)\), we denote by
\[
\mathcal{S}(u_0,u_1):=\Bigl\{v_t: \text{subgeodesic},\ \limsup_{t\to 0^+}v_t\leq u_0,\ \limsup_{t\to 1^-}v_t\leq u_1\Bigr\}
\]
the family of subgeodesics with the prescribed boundary conditions. We define the \textit{weak (Mabuchi) geodesic} joining \(u_0\) and \(u_1\) by the upper envelope
\begin{equation}\label{weak geodesic}
u_t(x):=\sup_{\mathcal{S}(u_0,u_1)} v_t(x).
\end{equation}

It is well known that for each \(v_t\in \mathcal{S}(u_0,u_1)\), the map \(t\mapsto v_t\) is convex, since \(V\) is \(\pi_X^*\theta\)-psh and invariant under imaginary translations (i.e., \(i\R\)-invariant); cf.\ \cite[Theorem~I.5.13]{agbook}. In particular,
\[
v_t\leq (1-t)u_0+tu_1.
\]
Taking the supremum in \eqref{weak geodesic}, we conclude that \(t\mapsto u_t\) is convex and satisfies
\begin{equation}\label{convexity of geodesic}
u_t\leq (1-t)u_0+tu_1,\qquad \forall t\in [0,1].
\end{equation}
In particular, \(\mathrm{usc}(u_t)\leq (1-t)u_0+tu_1\) for all \(t\in[0,1]\), hence \(t\mapsto \mathrm{usc}(u_t)\in \mathcal{S}(u_0,u_1)\), and \(u_t\) is upper semicontinuous in \(x\). Equivalently, the complexification \(U(x,\tau):=u_{\Rep \tau}(x)\) belongs to \(\PSH(X\times T,\pi_X^*\theta)\cup\{-\infty\}\).

We record some basic facts when the endpoints are bounded; see \cite{DDL18L1,DDL18}.

\begin{lem}\label{minimal singular/finite energy geodesic}
If \(u_0,u_1\in \PSH(X,\theta)\cap L^\infty(X)\) and \(t\mapsto u_t\) is the weak geodesic joining them, then
\begin{itemize}
\item[(i)] \(\lim_{t\to 0^+}u_t=u_0\) and \(\lim_{t\to 1^-}u_t=u_1\);
\item[(ii)] \(u_t\in L^\infty(X)\) for all \(t\in[0,1]\);
\item[(iii)] there exists \(C=C(u_0,u_1)>0\) such that \(|u_t-u_{t'}|\leq C|t-t'|\) for all \(t,t'\in[0,1]\).
\end{itemize}
\end{lem}

\begin{proof}
Assume \(u_0,u_1\in \PSH(X,\theta)\cap L^\infty(X)\), and choose \(C>0\) such that \(u_0-C\leq u_1\leq u_0+C\). Consider the path
\[
[0,1]\ni t\longmapsto v_t:=\max(u_0-Ct,\ u_1-C(1-t)).
\]
Then \(v_t\in \mathcal{S}(u_0,u_1)\), and by \eqref{weak geodesic} and \eqref{convexity of geodesic},
\[
v_t\leq u_t\leq (1-t)u_0+tu_1,\qquad \forall t\in (0,1).
\]
Assertions (i)--(iii) follow immediately.
\end{proof}

\begin{thm}[\cite{DDL18L1}, Theorems~2.5 and~3.13; \cite{DDL18}, Theorem~3.12]\label{geodesic and MA energy}
If \([0,1]\ni t\mapsto u_t\) is a weak geodesic joining \(u_0,u_1\in \E^1(X,\theta)\), then \(E_\theta\) is affine in \(t\), i.e.,
\[
E_\theta(u_t)=(1-t)E_\theta(u_0)+tE_\theta(u_1).
\]
\end{thm}

By convexity of \(t\mapsto u_t(x)\), the following limit exists:
\[
\dot{u}_0:=\lim_{t\to 0^+}\frac{u_t-u_0}{t}.
\]
We record the following useful estimate, which is similar to \cite[(6.6)]{BBGZ} and \cite[Proposition~3.2]{BF14}.
\begin{lem}\label{MA energy variation lemma}
Let \(u_0,u_1\in \E^1(X,\theta)\) satisfy \(u_1\leq u_0+C\) for some \(C\in \R\), and let \(t\mapsto u_t\) be the weak geodesic joining \(u_0\) and \(u_1\). Then \(\dot{u}_0\in L^1(X,\theta_{u_0}^n)\) and
\begin{equation}\label{MA energy variation estimate}
E_\theta(u_1)-E_\theta(u_0)=\frac{d}{dt}E_\theta(u_t)\Big|_{t=0}\leq\frac{1}{\vol(L)}\int_X \dot{u}_0\, \theta_{u_0}^n.
\end{equation}
\end{lem}

\begin{proof}
By convexity of \(t\mapsto u_t\), for any \(t\in (0,1]\), outside the pluripolar set \(\{u_0=-\infty\}\),
\[
\dot{u}_0(x)\leq \frac{u_t(x)-u_0(x)}{t}\leq u_1(x)-u_0(x)\leq C,\qquad \forall x\in X\setminus \{u_0=-\infty\}.
\]
On the other hand, by \eqref{Energy estimate}, we obtain
\[
\frac{E_\theta(u_t)-E_\theta(u_0)}{t}\leq \frac{1}{\vol(L)}\int_X\frac{u_t-u_0}{t}\,\theta_{u_0}^n,\qquad \forall t\in [0,1].
\]
By Theorem~\ref{geodesic and MA energy},
\[
\frac{d}{dt}E_\theta(u_t)\Big|_{t=0}=E_\theta(u_1)-E_\theta(u_0)>-\infty.
\]
This shows that \(\dot{u}_0\in L^1(X,\theta_{u_0}^n)\). Moreover, since
\[
\frac{u_t-u_0}{t}\downarrow \dot{u}_0,
\]
applying the monotone convergence theorem to
\[
C-\frac{u_t-u_0}{t}\uparrow C-\dot{u}_0
\]
gives
\[
\lim_{t\to 0^+}\int_X \left(C-\frac{u_t-u_0}{t}\right)\theta_{u_0}^n=\int_X (C-\dot{u}_0)\theta_{u_0}^n.
\]
By the integrability of \(\dot{u}_0\) with respect to \(\theta_{u_0}^n\), we obtain
\[
\frac{d}{dt}E_\theta(u_t)\Big|_{t=0}:=\lim_{t\to 0^+}\frac{E_\theta(u_t)-E_\theta(u_0)}{t}
\leq \lim_{t\to 0^+}\frac{1}{\vol(L)}\int_X\frac{u_t-u_0}{t}\,\theta_{u_0}^n
=\frac{1}{\vol(L)}\int_X\dot{u}_0\,\theta_{u_0}^n.
\]
\end{proof}

We also recall the following closely related differentiability result.

\begin{lem}[\cite{BB10}, Theorem~B; \cite{BBGZ}, Lemma~4.2]\label{variation formula of MA energy}
If \(u\in \E^1(X,\theta)\) and \(f\in C^0(X)\), then
\begin{equation}
\frac{d}{dt}\Big|_{t=0}E_\theta(P_\theta(u+tf))=\frac{1}{\vol(L)}\int_X f\,\theta_u^n.
\end{equation}
\end{lem}
\section{Bergman Measures and Equilibrium in the Adjoint Setting}\label{Bergman Measures, Positivity of Direct Images, and Equilibrium}
Throughout this section, we will work under the set-up of \eqref{set-up}. Recall that in the introduction, for a smooth volume form $\nu$ and a function $u:X\to [-\infty,\infty)$, we define the Hilbert norm as in \eqref{Hilbert norm}:
\[
 (s_1,s_2)_{\Hilb_k(u)}:=\int_X (h_0^k\otimes \nu^{-1})(s_1,s_2)e^{-ku}d\nu
\]
and we say $u$ is admissible if $\Hilb_k(u)$ is finite on whole $H^0(X,L^k\otimes K_X)$. Clearly, any $u\in \PSH(X,\theta)\cap L^\infty$ is admissible. Moreover,  by \cite[Theorem 1.1]{DDL18}, any $u\in \E^1(X,\theta)$ has zero Lelong number and thus is admissible. 
\begin{lem}\label{Hilbert norm independent of the choice of volume form} For each fixed $k\in \N$, $\Hilb_k$ and  $E_k$ satisfy:
\begin{itemize}
\item[(i)] $\Hilb_k$ is independent of the choice of volume form $\nu$.
    \item[(ii)] if $u_1,u_2$ are admissible and $u_1\leq u_2$, then $\Hilb_k(u_1)\geq \Hilb_k(u_2)$ and $E_k(u_1)\leq E_k(u_2)$;
\end{itemize}
\end{lem}
\begin{proof}
    For (i), we can choose a holomorphic frame $\sigma$ for $K_X$ and $e_L$ for $L$ over an open set $U\subset X$. Then $h_0(e_L,e_L)=e^{-\Phi_0}$ for some $\Phi_0\in C^\infty(U)$ and  $
    |\sigma|_{\nu^{-1}}^2=i^{n^2}\sigma\wedge \bar{\sigma}/\nu$. For $s\in H^0(X,L^k\otimes K_X)$, we can write $s=fe_L^k\otimes \sigma$ for some $f\in \OO(U)$. Then 
    \[
    h_0^k\otimes \nu^{-1}(s,s)e^{-ku}d\nu=|f|^2e^{-k(\Phi_0+u)}\frac{i^{n^2}\sigma\wedge \bar{\sigma}}{\nu}d\nu=|f|^2e^{-k(\Phi_0+u)}i^{n^2}\sigma\wedge \bar{\sigma}.
    \]
As a result, $\|s\|^2_{\Hilb_k(u)}$ is independent of the choice of $\nu$. (ii) follows immediately from definition.
\end{proof}
Recall that we define Bergman kernel and Bergman measure in \eqref{Bergman measure}. By standard extremal characterization of Bergman kernel, we see that
\begin{equation}\label{extremal characterization}
 K^k_{u,\nu}=\sup\{(h_0^k\otimes \nu^{-1})(s,s):s\in H^0(X,L^k\otimes K_X),\|s\|_{\Hilb_k(u)}\leq 1\}.   
\end{equation}
Hence, Bergman kernel (resp., Bergman measure) is independent of the choice of $\Hilb_k(u)$-orthonormal basis on $H^0(X,L^k\otimes K_X)$. Moreover, as in the proof of Lemma \ref{Hilbert norm independent of the choice of volume form}, the Bergman measure $B^k_u$ is independent of the choice of smooth volume form $\nu$ and 
\[
\int_X B^k_u=\int_X K^k_{u,\nu}e^{-ku}d\nu=\sum_{j=1}^{N_k}\|S_j\|_{\Hilb_k(u)}^2=N_k.
\]
We also need the quantized analogue of \eqref{MA energy variation estimate} for the quantized Monge--Amp\`ere energy. The proof is inspired by \cite[Theorem~6.6]{BBGZ} and \cite[Lemma~5.1]{BB10}.

\begin{lem}\label{QMA variation}
Let \(u_0,u_1\in\PSH(X,\theta)\cap L^\infty(X)\), and let \([0,1]\ni t\mapsto u_t\) be the weak geodesic joining them. Then
\begin{equation}\label{QMA variation formula}
\frac{d}{dt}\Big|_{t=0^+}E_k(u_t)=\frac{1}{N_k}\int_X \dot{u}_0\, B^k_{u_0}.
\end{equation}
\end{lem}

\begin{proof}
Let \(\{S_1,\dots,S_{N_k}\}\) be an orthonormal basis of \(V_k:=H^0(X,L^k\otimes K_X)\) with respect to \(\Hilb_k(u_0)\). For \(t>0\), set \(H_k(t):=\Hilb_k(u_t)\), viewed in this basis as the Gram matrix
\[
H_k(t)_{ij}:=(S_i,S_j)_{\Hilb_k(u_t)}=\int_X (h_0^k\otimes \nu^{-1})(S_i,S_j)\,e^{-ku_t}\,d\nu.
\]

We first show that \(E_k(u_t)\) is right-differentiable at \(t=0\). By Lemma~\ref{minimal singular/finite energy geodesic}(iii), there exists \(C>0\) such that \(\sup_X|u_t-u_0|\le C t\) for all \(t\in[0,1]\). For each \(i,j\),
\[
\frac{H_k(t)_{ij}-H_k(0)_{ij}}{t}
=\int_X (h_0^k\otimes \nu^{-1})(S_i,S_j)\,e^{-ku_0}\,
\frac{e^{-k(u_t-u_0)}-1}{t}\,d\nu.
\]
By the mean value theorem,
\[
\frac{e^{-k(u_t(x)-u_0(x))}-1}{t}
=-k e^{-k\xi(t,x)}\,\frac{u_t(x)-u_0(x)}{t},
\]
for some \(\xi(t,x)\) between \(0\) and \(u_t(x)-u_0(x)\). Since \(|\xi(t,x)|\le |u_t(x)-u_0(x)|\le Ct\to 0\) as \(t\to 0^+\), we have pointwise
\[
\lim_{t\to 0^+}\frac{e^{-k(u_t(x)-u_0(x))}-1}{t}=-k\dot{u}_0(x),
\]
and the integrand is uniformly dominated by a constant multiple of
\((h_0^k\otimes \nu^{-1})(S_i,S_j)e^{-ku_0}\) (using boundedness of \(\frac{u_t-u_0}{t}\) from Lemma~\ref{minimal singular/finite energy geodesic}(iii)). Hence, dominated convergence yields
\[
\lim_{t\to 0^+}\frac{H_k(t)_{ij}-H_k(0)_{ij}}{t}
=-k\int_X (h_0^k\otimes \nu^{-1})(S_i,S_j)\,\dot{u}_0\,e^{-ku_0}\,d\nu.
\]

Since \(H_k(0)=I_{N_k}\), we have
\[
\frac{d}{dt}\Big|_{t=0^+}\log\det H_k(t)
=\tr\!\left(\frac{d}{dt}\Big|_{t=0^+}H_k(t)\right)
=-k\sum_{j=1}^{N_k}\int_X (h_0^k\otimes \nu^{-1})(S_j,S_j)\,\dot{u}_0\,e^{-ku_0}\,d\nu.
\]
Using \(B^k_{u_0}=K^k_{u_0,\nu}e^{-ku_0}d\nu=\sum_{j=1}^{N_k}(h_0^k\otimes \nu^{-1})(S_j,S_j)e^{-ku_0}d\nu\) and the definition of \(E_k\), we obtain \eqref{QMA variation formula}.
\end{proof}

The positivity of direct images after Berndtsson--P\u{a}un \cite[Theorem~3.5]{BP08} (see also \cite[Proposition~3.4(ii)]{BF14}) implies the following.

\begin{thm}\label{convexity of quantized Monge--Ampere energy along geodesic}
Let \(t\mapsto u_t\) be the weak geodesic joining \(u_0,u_1\in \PSH(X,\theta)\cap L^\infty(X)\). Then \(t\mapsto E_k(u_t)\) is convex on \([0,1]\).
\end{thm}

Similarly, we have the following differentiability statement for \(E_k\), due to \cite[Lemma~5.1]{BB10} and \cite[Proposition~2.4]{BBWN}, adapted to our setting.

\begin{lem}\label{variational formula for QMA along affine path}
If \(\phi,f\in L^\infty(X)\), then \(t\mapsto E_k(\phi+tf)\) is concave and differentiable at \(t=0\), with
\begin{equation}\label{QMA affine variation}
\frac{d}{dt}\Big|_{t=0}E_k(\phi+tf)=\frac{1}{N_k}\int_X f\,B^k_{\phi}.
\end{equation}
\end{lem}

\begin{proof}
The differentiability is proved in \cite[Lemma~5.1]{BB10} for smooth \(\phi,f\). The only regularity needed there is the ability to differentiate under the integral sign, which still holds for bounded \(\phi,f\) by dominated convergence. Concavity follows from \cite[Proposition~2.4]{BBWN} (see also the direct Bergman-kernel computation in \cite[Lemma~2.2]{BW08}).
\end{proof}

We now recall Berman's \textit{local Morse inequality}, as stated in \cite[Lemma~4.1]{Ber09} (see \cite[Theorem~2.1]{BoSurvey} and \cite{BermanlocalMorse} for proofs).

\begin{thm}\label{Berman local Morse inequality}
Under the set-up \eqref{set-up}, for \(\phi\in C^\infty(X)\),
\begin{itemize}
\item[(i)] there exists a constant \(C>0\), independent of \(k\) and \(\phi\), such that
\[
\frac{n!}{k^n}K^k_{\phi,\nu}\, e^{-k\phi}\leq C\qquad \text{on }X;
\]
\item[(ii)] we have the pointwise asymptotic estimate on \(X\):
\[
\limsup_{k\to\infty}\frac{n!}{k^n}K_{\phi,\nu}^k(x)\,e^{-k\phi(x)}
\leq \mathbf{1}_{X_+(\phi)}(x)\,\frac{\theta_\phi^n}{\nu}(x),
\]
where \(X_+(\phi):=\{x\in X:\theta_\phi(x)>0\}\).
\end{itemize}
\end{thm}

We now prove an analogue of \cite[Lemma~4.2 and Theorem~4.6]{Ber09}.

\begin{thm}[Global domination and weak convergence]\label{prop:bergman-convergence}
Under the set-up \eqref{set-up}, for \(\phi\in C^\infty(X)\) there exist constants \(C_1=C_1(\nu,\omega)>0\) and \(C_2=C_2(\nu,\omega)>0\) such that for all \(k\),
\begin{equation}\label{global domination}
\frac{n!}{k^n}K^k_{\phi,\nu}
\leq C_2\,\exp\!\Big(k\big(P_{\theta+\delta_k\omega}(\phi)-\phi\big)\Big)
\quad \text{on }X,
\qquad \delta_k:=\frac{C_1}{k}.
\end{equation}
Moreover, \(N_k^{-1}B^k_{\phi}\) converges weakly to the equilibrium measure \(\vol(L)^{-1}\theta_{P_\theta(\phi)}^n\).
\end{thm}

\begin{proof}
Fix an \(\Hilb_k(\phi)\)-orthonormal basis \(\{S_1,\dots,S_{N_k}\}\) of \(H^0(X,L^k\otimes K_X)\), and consider the Kodaira map
\[
\mathrm{Kod}_k:U:=X\setminus Z \longrightarrow \P^{N_k-1},\qquad
x\longmapsto [S_1(x):\cdots:S_{N_k}(x)],
\]
where \(Z\) is the base locus of \(L^k\otimes K_X\). Then \(\frac{1}{k}\log K_{\phi,\nu}^k\in C^\infty(U)\), and on \(U\),
\[
\frac{1}{k}dd^c\log K^k_{\phi,\nu}
=-\frac{1}{k}c_1(L^k\otimes K_X,h_0^k\otimes \nu^{-1})
+\frac{1}{k}\mathrm{Kod}_k^*\omega_{FS},
\]
where \(\omega_{FS}\) is the Fubini--Study form on \(\P^{N_k-1}\). Since
\[
\frac{1}{k}c_1(L^k\otimes K_X,h_0^k\otimes \nu^{-1})
=\theta+\frac{1}{k}c_1(K_X,\nu^{-1}),
\]
there exists \(C_1=C_1(\nu,\omega)>0\) such that
\[
\theta+\frac{1}{k}c_1(K_X,\nu^{-1})\geq \theta-\delta_k\omega,\qquad \delta_k=\frac{C_1}{k}.
\]
Consequently, \(\frac{1}{k}\log K_{\phi,\nu}^k\in \PSH(X,\theta+\delta_k\omega)\) in the sense of currents, i.e.,
\[
\theta+\delta_k\omega+\frac{1}{k}dd^c\log K_{\phi,\nu}^k\geq 0.
\]
By Theorem~\ref{Berman local Morse inequality}(i), there exists \(C>0\) such that
\[
\frac{1}{k}\log K_{\phi,\nu}^k-\phi
\leq \frac{1}{k}\log\!\Big(\frac{Ck^n}{n!}\Big)
=\frac{\log C-\log n!}{k}+\frac{n\log k}{k}.
\]
Using the translation property of envelopes,
\[
P_{\theta+\delta_k\omega}(\phi+C')=P_{\theta+\delta_k\omega}(\phi)+C',
\]
we obtain
\[
\frac{1}{k}\log K_{\phi,\nu}^k
\leq P_{\theta+\delta_k\omega}(\phi)+\frac{\log C-\log n!}{k}+\frac{n\log k}{k},
\]
which implies \eqref{global domination} after exponentiating.

Next, note that \(\delta_k\downarrow 0\) implies \(P_{\theta+\delta_k\omega}(\phi)\downarrow P_\theta(\phi)\). Let \(D:=\{\phi=P_\theta(\phi)\}\) be the contact set. For \(\epsilon>0\), consider the open set
\[
U_\epsilon:=\{P_\theta(\phi)+\epsilon<\phi\}\subset X.
\]
For each fixed \(x\in U_\epsilon\), we have
\[
P_{\theta+\delta_k\omega}(\phi)(x)-\phi(x)\leq -\epsilon/2
\]
for all sufficiently large \(k\). Combining this with \eqref{global domination} yields
\[
\frac{n!}{k^n}K^k_{\phi,\nu}(x)e^{-k\phi(x)}\longrightarrow 0
\qquad (x\in U_\epsilon).
\]
Using \eqref{stability of volume}, we have
\[
\lim_{k\to\infty}\frac{k^n/n!}{N_k}=\frac{1}{\vol(L)},
\]
and dominated convergence therefore gives
\[
\lim_{k\to\infty}\int_{U_\epsilon} \frac{1}{N_k}B_\phi^k
=\lim_{k\to\infty}\frac{1}{N_k}\int_{U_\epsilon} K^k_{\phi,\nu}e^{-k\phi}\,d\nu
=0.
\]
By the Banach--Alaoglu theorem, there exist a subsequence \(k_\ell\to\infty\) and a probability measure \(\beta\) such that
\[
N_{k_\ell}^{-1}B^{k_\ell}_\phi\rightharpoonup \beta.
\]
Since \(U_\epsilon\) is open, the Portmanteau theorem yields
\[
\beta(U_\epsilon)\leq \liminf_{\ell\to\infty}\int_{U_\epsilon}\frac{1}{N_{k_\ell}}B_\phi^{k_\ell}=0.
\]
Thus \(\beta(U_\epsilon)=0\) for every \(\epsilon>0\). Since
\[
X\setminus D=\{\phi>P_\theta(\phi)\}=\bigcup_{m=1}^\infty U_{1/m},
\]
it follows that \(\beta(X\setminus D)=0\). Hence \(\supp\beta\subset D\).

Finally, let \(\psi\in C^0(X)\) be nonnegative. Then Theorem~\ref{Berman local Morse inequality}(ii) and dominated convergence give
\[
\limsup_{\ell\to\infty}\frac{n!}{k_\ell^n}\int_X \psi\, K^{k_\ell}_{\phi,\nu} e^{-k_\ell\phi}\,d\nu
\leq \int_X \psi\,\mathbf{1}_{X_+(\phi)}\,\theta_\phi^n.
\]
Moreover, by  \eqref{Concentration of equilibrium measure} and \eqref{contact set is contained in semipositive set},
\[
\theta_{P_\theta(\phi)}^n=\mathbf{1}_D\,\theta_\phi^n=\mathbf{1}_{X_+(\phi)}\,\theta_\phi^n.
\]
Hence
\[
\int_X\psi\,d\beta\leq \int_X\psi\,\frac{1}{\vol(L)}\,\theta_{P_\theta(\phi)}^n.
\]
Since \(\beta\) is a probability measure supported on \(D\), this forces
\[
\beta=\vol(L)^{-1}\theta_{P_\theta(\phi)}^n.
\]
Therefore the whole sequence \(N_k^{-1}B^k_\phi\) converges weakly to \(\vol(L)^{-1}\theta_{P_\theta(\phi)}^n\).
\end{proof}

Finally, we recall a well-known implication proved in \cite[Theorem~A]{BB10}. The argument carries over verbatim to the adjoint setting (compare also \cite[Prop.~6.4]{DX24}).

\begin{cor}\label{Quantization of Monge-Ampere energy for smooth functions}
If \(L\) is big and \(\phi\in C^\infty(X)\), then
\[
\lim_{k\to\infty} E_k(\phi)=E_\theta\big(P_\theta(\phi)\big).
\]
\end{cor}

\begin{proof}
Applying Lemma~\ref{variational formula for QMA along affine path} to the affine path \(t\mapsto t\phi\), we obtain
\[
E_k(\phi)=\int_0^1 \frac{d}{dt}E_k(t\phi)\,dt
=\int_0^1 \left(\frac{1}{N_k}\int_X \phi\,B^k_{t\phi}\right)dt.
\]
By Theorem~\ref{prop:bergman-convergence}, for each fixed \(t\in[0,1]\), \(N_k^{-1}B^k_{t\phi}\) converges weakly to \(\vol(L)^{-1}\theta_{P_\theta(t\phi)}^n\). Hence,
\[
\lim_{k\to\infty}\frac{1}{N_k}\int_X \phi\,B^k_{t\phi}
=\frac{1}{\vol(L)}\int_X \phi\,\theta_{P_\theta(t\phi)}^n
=\frac{d}{dt}E_\theta\big(P_\theta(t\phi)\big),
\]
where we used Lemma~\ref{variation formula of MA energy} in the last identity. Since each \(N_k^{-1}B^k_{t\phi}\) is a probability measure and \(\left|\frac{1}{N_k}\int_X \phi\,B^k_{t\phi}\right|\le \|\phi\|_{L^\infty(X)}\) uniformly in \(k,t\), the result follows from dominated convergence.
\end{proof}
\section{Peak Sections and Bergman Kernel Comparison}\label{peak section and Bergman Kernel Comparison}
We continue to work under the set-up \eqref{set-up}. Since \(X\) is projective, we fix an ample line bundle \(A\) on \(X\), equipped with a smooth Hermitian metric \(g\) such that \(\omega:=c_1(A,g)\) is a K\"ahler form. We also take the smooth volume form \(\nu:=\omega^n/n!\) induced by \(\omega\).

For \(p\in\N\), we further consider the twisted spaces \(H^0(X,L^k\otimes K_X\otimes A^{p})\). Given a function \(u:X\to[-\infty,\infty)\), we define the corresponding Hilbert norm by
\[
\Hilb_{k,p}(u)(\sigma,\sigma)=\|\sigma\|_{\Hilb_{k,p}(u)}^{2}
:=\int_X (h_0^k\otimes \nu^{-1}\otimes g^{p})(\sigma,\sigma)\,e^{-ku}\,d\nu,
\qquad \sigma\in H^0(X,L^k\otimes K_X\otimes A^{p}),
\]
whenever the integral is finite. We say that \(u\) is \textit{admissible} (for \((k,p)\)) if \(\Hilb_{k,p}(u)\) is finite on the whole space \(H^0(X,L^k\otimes K_X\otimes A^{p})\).

For such an admissible function \(u\), we similarly define the Bergman kernel (resp.\ Bergman measure) by
\[
K^{k,p}_{u,\nu}:=\sum_{j=1}^{N_{k,p}}(h_0^k\otimes \nu^{-1}\otimes g^{p})(\sigma_j,\sigma_j),
\qquad
B_{u}^{k,p}:=K_{u,\nu}^{k,p}\,e^{-ku}\,d\nu,
\]
where \(\{\sigma_1,\dots,\sigma_{N_{k,p}}\}\) is an \(\Hilb_{k,p}(u)\)-orthonormal basis and \(N_{k,p}:=h^0(X,L^k\otimes K_X\otimes A^p)\).
As in Lemma~\ref{Hilbert norm independent of the choice of volume form}, \(\Hilb_{k,p}\) and \(B_{u}^{k,p}\) are independent of the choice of volume form \(\nu\). Moreover, by the extremal characterization, \(K_{u,\nu}^{k,p}\) and \(B_u^{k,p}\) are independent of the choice of \(\Hilb_{k,p}(u)\)-orthonormal basis:
\begin{equation}\label{extremal characterization for twist}
K_{u,\nu}^{k,p}(x)
=\sup\Bigl\{(h_0^k\otimes \nu^{-1}\otimes g^{p})(\sigma,\sigma)(x):
\sigma\in H^0(X,L^k\otimes K_X\otimes A^{p}),\ \|\sigma\|_{\Hilb_{k,p}(u)}\leq 1\Bigr\}.
\end{equation}

For \(p\in\N\), we also denote by \(\mathrm{H}_p^A\) the Hilbert norm on \(H^0(X,A^p)\) given by
\[
H_p^A(t,t)=\|t\|_{H_p^A}^2:=\int_X g^p(t,t)\,\frac{\omega^n}{n!}.
\]
For an \(\mathrm{H}_p^A\)-orthonormal basis \(\{t_1,\dots,t_{d_p}\}\), we define the Bergman kernel of \(A^p\) by
\begin{equation}\label{Bergman kernel for ample line bundle}
K^p_A(x):=\sum_{j=1}^{d_p} g^p(t_j,t_j)(x)
=\sup\{g^p(s,s)(x):\|s\|_{\mathrm{H}_p^A}\leq 1\},
\qquad d_p:=h^0(X,A^p).
\end{equation}

We will need a pointwise semi-classical Ohsawa--Takegoshi type statement. It follows, for instance, from the general results of Finski \cite[Theorem~1.10]{Finski1} or from off-diagonal asymptotics of Bergman kernels \cite{DLM,MM2}. We give a direct proof using the (on-diagonal) Bergman kernel asymptotics (see also \cite[Proposition~5.2]{BD26}).

\begin{lem}\label{semi-classical Ohsawa--Takegoshi}
There exists a sequence \(D_p\downarrow 1\) such that for any \(x\in X\), there exists \(t_x\in H^0(X,A^p)\) satisfying
\[
g^p(t_x,t_x)(x)=1,
\qquad
\sup_X g^p(t_x,t_x)\leq D_p^2.
\]
\end{lem}

\begin{proof}
By the diagonal Bergman kernel asymptotics \cite{Tian,Bouche,Catlin,Zelditch}, we have
\[
p^{-n}K_A^p\to 1\quad \text{uniformly on }X.
\]
Hence, there exist \(D_p\downarrow 1\) and \(p_0\in\N\) such that for all \(p\geq p_0\),
\[
\frac{p^n}{D_p}\leq K_A^p(y)\leq p^n D_p,\qquad \forall y\in X.
\]
By \eqref{Bergman kernel for ample line bundle}, we can choose \(\widetilde t_x\in H^0(X,A^p)\) with \(\|\widetilde t_x\|_{\mathrm{H}_p^A}=1\) and
\[
g^p(\widetilde t_x,\widetilde t_x)(x)=K_A^p(x).
\]
By extremal characterization, \(g^p(\widetilde t_x,\widetilde t_x)(y)\leq K_A^p(y)\) for all \(y\in X\). Setting
\[
t_x:=\frac{\widetilde t_x}{\sqrt{K_A^p(x)}},
\]
we have \(g^p(t_x,t_x)(x)=1\) and, for all \(y\in X\),
\[
g^p(t_x,t_x)(y)\leq \frac{K_A^p(y)}{K_A^p(x)}\leq D_p^2.
\]
\end{proof}

The main result of this section is the following comparison theorem between Bergman kernels, which is crucial for the proof of Theorem~\ref{quantization of MA energy}.

\begin{thm}\label{Bergman kernel comparison}
For \(u\in \PSH(X,\theta)\cap L^\infty(X)\) and any \(k\in\N\), there exists a sequence \(D_p\downarrow 1\) such that
\[
K^{k}_{u,\nu}\leq D_p^2\,K^{k,p}_{u,\nu},\quad p\gg 1.
\]
\end{thm}

\begin{proof}
Fix \(x\in X\), and let \(t_x\in H^0(X,A^p)\) be given by Lemma~\ref{semi-classical Ohsawa--Takegoshi}, so that
\[
g^p(t_x,t_x)(x)=1,
\qquad
M_p:=\sup_X g^p(t_x,t_x)\leq D_p^2.
\]
Let \(\sigma\in H^0(X,L^k\otimes K_X)\) satisfy \(\|\sigma\|_{\Hilb_k(u)}\le 1\). Define
\[
\tau:=M_p^{-1/2}\,\sigma\otimes t_x\in H^0\!\left(X,L^k\otimes K_X\otimes A^p\right).
\]
Then
\begin{align*}
\|\tau\|_{\Hilb_{k,p}(u)}^2
&=\int_X (h_0^k\otimes \nu^{-1}\otimes g^p)(\tau,\tau)\,e^{-ku}\,d\nu =\frac{1}{M_p}\int_X (h_0^k\otimes \nu^{-1})(\sigma,\sigma)\,g^p(t_x,t_x)\,e^{-ku}\,d\nu \\
&\le \int_X (h_0^k\otimes \nu^{-1})(\sigma,\sigma)\,e^{-ku}\,d\nu
\le 1.
\end{align*}
Since \(g^p(t_x,t_x)(x)=1\), the extremal characterization \eqref{extremal characterization for twist} yields
\begin{align*}
K_{u,\nu}^{k,p}(x)
&\ge (h_0^k\otimes \nu^{-1}\otimes g^p)(\tau,\tau)(x) \\
&=\frac{1}{M_p}(h_0^k\otimes \nu^{-1})(\sigma,\sigma)(x)\,g^p(t_x,t_x)(x)
=\frac{1}{M_p}(h_0^k\otimes \nu^{-1})(\sigma,\sigma)(x).
\end{align*}
Taking the supremum over \(\sigma\) with \(\|\sigma\|_{\Hilb_k(u)}\le 1\) gives
\(
K^k_{u,\nu}(x)\le M_p\,K_{u,\nu}^{k,p}(x)\le D_p^2 K_{u,\nu}^{k,p}(x),
\)
as claimed.
\end{proof}

We will also need an \(\epsilon\)-twisted analogue of the local Morse inequality.

\begin{prop}\label{eps twisted local morse}
Fix \(\epsilon>0\), and set \(p_k:=\lfloor \epsilon k\rfloor\). Then
\begin{equation}\label{p-twisted-local-holomorphic-morse-eps}
\limsup_{k\to\infty}N_k^{-1}B_0^{k,p_k}
\leq \frac{(\theta+\epsilon\omega)^n}{\vol(L)},
\qquad N_k:=h^0(X,L^k\otimes K_X).
\end{equation}
\end{prop}

\begin{proof}
Let \(K_{0,\nu}^{k,p_k}\) be the Bergman kernel of \(H^0(X,L^k\otimes K_X\otimes A^{p_k})\) defined above, with \(\nu=\omega^n/n!\). For \(x\in X\), choose a coordinate polydisk \(P(x,r)\subset X\). On \(P(x,r)\), write
\[
h_0=e^{-\Phi_0},\qquad g=e^{-\eta},\qquad \nu=e^{-\rho}\,d\lambda,
\]
where \(d\lambda\) denotes Lebesgue measure in these coordinates. It is well known (see, e.g., \cite[Lemma~2.3]{BD26}) that
\[
K^{k,p_k}_{0,\nu}(x)
\leq K_{P(x,r),\,k\Phi_0+p_k\eta+\rho}(x)\,e^{-(k\Phi_0(x)+p_k\eta(x)+\rho(x))},
\]
where \(K_{P(x,r),\,k\Phi_0+p_k\eta+\rho}\) denotes the local Bergman kernel for \(L^2\)-holomorphic functions on \(P(x,r)\) with weight \(e^{-(k\Phi_0+p_k\eta+\rho)}\).
By the standard submean inequality argument (see \cite[Lemma~3.1(ii)]{Berman09b}, and also \cite[Proposition~2.5]{BF14}), we obtain
\[
\limsup_{k\to\infty} k^{-n}\,K_{P(x,r),\,k\Phi_0+p_k\eta+\rho}(x)\,
e^{-(k\Phi_0(x)+p_k\eta(x)+\rho(x))}
\leq \frac{(\theta+\epsilon\omega)^n}{\omega^n}(x),
\]
since \((\theta+\epsilon\omega)(x)>0\) for all \(x\in X\) (as \(\theta\ge 0\) and \(\omega>0\)).

Finally, using \eqref{stability of volume} (to identify the asymptotic \(\frac{k^n/n!}{N_k}\to \frac{1}{\vol(L)}\)), we conclude
\[
\limsup_{k\to\infty}N_k^{-1}B^{k,p_k}_0
=\limsup_{k\to\infty}\frac{k^n/n!}{N_k}\,k^{-n}K^{k,p_k}_{0,\nu}\,\omega^n
\leq \frac{1}{\vol(L)}(\theta+\epsilon\omega)^n,
\]
which is \eqref{p-twisted-local-holomorphic-morse-eps}.
\end{proof}\section{Proofs of Theorem~\ref{quantization of MA energy} and Corollary~\ref{convergence of Bergman measures}}
Recall that \(X\) is a compact projective manifold of dimension \(n\), and \(L\) is a big and semipositive line bundle over \(X\), equipped with a fixed smooth Hermitian metric \(h_0\) whose Chern curvature form \(\theta:=c_1(L,h_0)\) satisfies \(\theta\geq 0\). We first deduce Corollary~\ref{convergence of Bergman measures} from Theorem~\ref{quantization of MA energy}. The proof is inspired by Berman--Boucksom--Witt Nystr\"om \cite{BBWN}; see also \cite[Theorem~4.3]{BD26}.

\begin{proof}[Proof of Corollary~\ref{convergence of Bergman measures}]
Fix \(u\in \PSH(X,\theta)\cap L^\infty(X)\). Let \(f\in C^0(X)\) and \(t\in\R\). Then \(u+tf\) is bounded. By Lemma~\ref{variational formula for QMA along affine path}, the function \(t\mapsto E_k(u+tf)\) is concave and differentiable at \(t=0\), with
\[
\frac{d}{dt}\Big|_{t=0}E_k(u+tf)=\frac{1}{N_k}\int_X f\,B_u^k.
\]

By definition of the envelope,
\[
u+tf\ge P_\theta(u+tf)\ge \inf_X(u+tf)>-\infty,
\]
hence \(P_\theta(u+tf)\in \PSH(X,\theta)\cap L^\infty(X)\). Therefore Theorem~\ref{quantization of MA energy} applies and gives, for every \(t\in\R\),
\[
\liminf_{k\to\infty}E_k(u+tf)
\ge \lim_{k\to\infty}E_k\!\big(P_\theta(u+tf)\big)
=E_\theta\!\big(P_\theta(u+tf)\big),
\]
and in particular
\[
\lim_{k\to\infty}E_k(u)=E_\theta(u).
\]
On the other hand, by Lemma~\ref{variation formula of MA energy}, the map \(t\mapsto E_\theta(P_\theta(u+tf))\) is differentiable at \(t=0\), and
\[
\frac{d}{dt}\Big|_{t=0}E_\theta\!\big(P_\theta(u+tf)\big)
=\frac{1}{\vol(L)}\int_X f\,\theta_u^n.
\]

We now compare derivatives using concavity. Fix \(t>0\). By concavity of \(t\mapsto E_k(u+tf)\),
\[
\frac{1}{N_k}\int_X f\,B_u^k
=\frac{d}{dt}\Big|_{t=0}E_k(u+tf)
\ge \frac{E_k(u+tf)-E_k(u)}{t}.
\]
Taking \(\liminf_{k\to\infty}\) and using the limits above yields
\[
\liminf_{k\to\infty}\frac{1}{N_k}\int_X f\,B_u^k
\ge \frac{E_\theta(P_\theta(u+tf))-E_\theta(u)}{t}.
\]
Letting \(t\downarrow 0\), we obtain
\[
\liminf_{k\to\infty}\frac{1}{N_k}\int_X f\,B_u^k
\ge \frac{1}{\vol(L)}\int_X f\,\theta_u^n.
\]

Similarly, for \(t<0\), concavity gives
\[
\frac{1}{N_k}\int_X f\,B_u^k
\le \frac{E_k(u+tf)-E_k(u)}{t}.
\]
Taking \(\limsup_{k\to\infty}\) and letting \(t\uparrow 0\), we obtain
\[
\limsup_{k\to\infty}\frac{1}{N_k}\int_X f\,B_u^k
\le \frac{1}{\vol(L)}\int_X f\,\theta_u^n.
\]
Combining the two inequalities, we conclude that
\[
\lim_{k\to\infty}\frac{1}{N_k}\int_X f\,B_u^k
=\frac{1}{\vol(L)}\int_X f\,\theta_u^n.
\]
Since this holds for all \(f\in C^0(X)\), it follows that
\[
N_k^{-1}B_u^k\rightharpoonup \vol(L)^{-1}\theta_u^n.
\]
\end{proof}
We now turn to the upper bound in Theorem~\ref{quantization of MA energy}, which holds for any finite-energy potential; this is outlined in \cite[Theorem~3.5]{BF14}.

\begin{thm}\label{upper bound estimate}
For \(u\in \E^1(X,\theta)\), we have
\[
\limsup_{k\to\infty}E_k(u)\leq E_\theta(u).
\]
\end{thm}

\begin{proof}
Since \(u\) is upper semicontinuous, we can find a sequence \(\phi_j\in C^\infty(X)\) such that \(\phi_j\downarrow u\). Clearly, each \(\phi_j\) is admissible, and Lemma~\ref{Hilbert norm independent of the choice of volume form}(ii) gives
\[
E_k(u)\leq E_k(\phi_j).
\]
By Corollary~\ref{Quantization of Monge-Ampere energy for smooth functions},
\[
\limsup_{k\to\infty}E_k(u)
\leq \lim_{k\to\infty}E_k(\phi_j)
=E_\theta(P_\theta(\phi_j)),
\qquad \forall j\in\N.
\]
On the other hand,
\[
u\leq P_\theta(\phi_j)\leq \phi_j,
\]
hence \(P_\theta(\phi_j)\downarrow u\). By Theorem~\ref{Properties on Energy and E1}(ii),
\[
\limsup_{k\to\infty}E_k(u)\leq \lim_{j\to\infty}E_\theta(P_\theta(\phi_j))=E_\theta(u).
\]
\end{proof}

\begin{rmk}
The proof actually shows that Theorem~\ref{upper bound estimate} holds for any big and nef line bundle \(L\).
\end{rmk}

For the lower bound, we now assume that \(u\in \PSH(X,\theta)\cap L^\infty(X)\); see Remark~\ref{pitfall for finite energy potentials} for the obstruction in the finite-energy case.

\begin{proof}[Proof of Theorem~\ref{quantization of MA energy}]
It remains to prove that for any \(u\in \PSH(X,\theta)\cap L^\infty(X)\),
\begin{equation}\label{lower bound on energy}
\liminf_{k\to\infty}E_k(u)\geq E_\theta(u).
\end{equation}
Fix such a \(u\). Replacing \(u\) by \(u-\sup_X u\), we may assume that \(u\le 0\). Let \(t\mapsto u_t\) be the weak geodesic with \(u_0=0\) and \(u_1=u\). By Lemma~\ref{minimal singular/finite energy geodesic}(iii), we have \(\dot{u}_0\in L^\infty(X)\). By Theorem~\ref{convexity of quantized Monge--Ampere energy along geodesic}, \(E_k(u_t)\) is convex in \(t\), while Theorem~\ref{geodesic and MA energy} shows that \(E_\theta(u_t)\) is affine in \(t\). Hence, by convexity,
\[
E_k(u)-E_\theta(u)
\geq \lim_{t\to 0^+}\frac{E_k(u_t)-E_\theta(u_t)}{t}
\geq \int_X \dot{u}_0\left(\frac{1}{N_k}B_{0}^k-\frac{1}{\vol(L)}\theta^n\right),
\]
where we used \eqref{MA energy variation estimate} and Lemma~\ref{QMA variation} in the last inequality. Note that \(\dot{u}_0\) is integrable with respect to \(N_k^{-1}B_{0}^k\), since it is bounded.

Now fix \(\epsilon>0\). For each \(k\in\N\), set \(p_k:=\lfloor \epsilon k\rfloor\). Then the twisted bundle \(L^k\otimes K_X\otimes A^{p_k}\) is defined for every \(k\), and the associated Bergman measure \(B_0^{k,p_k}\) is defined as in Section~\ref{peak section and Bergman Kernel Comparison}. Since \(p_k\to\infty\) as \(k\to\infty\), Theorem~\ref{Bergman kernel comparison}, applied with \(p=p_k\), together with \eqref{p-twisted-local-holomorphic-morse-eps} yields a sequence \(D_{p_k}\downarrow 1\) such that
\begin{equation}\label{key estimate}
\frac{1}{N_k}B_0^k\le \frac{D_{p_k}^2}{N_k}\,B^{k,p_k}_0.
\end{equation}
Since \(u\le 0\), we have \(\dot u_0\le 0\). It follows from \eqref{key estimate} that
\begin{equation}\label{key estimate epsk}
E_k(u)-E_\theta(u)\geq \int_X\dot{u}_0\left(\frac{D_{p_k}^2}{N_k}B_0^{k,p_k}-\frac{1}{\vol(L)}\theta^n\right).
\end{equation}
By \eqref{p-twisted-local-holomorphic-morse-eps} and \(D_{p_k}\downarrow 1\), we obtain
\[
\limsup_{k\to\infty}\frac{D_{p_k}^2}{N_k}B_0^{k,p_k}
\leq \frac{(\theta+\epsilon\omega)^n}{\vol(L)}
\]
as measures. Since \(\dot u_0\in L^\infty(X)\), passing to the \(\liminf\) in \eqref{key estimate epsk} yields
\[
\liminf_{k\to\infty}\bigl(E_k(u)-E_\theta(u)\bigr)
\geq \frac{1}{\vol(L)}\int_X \dot u_0\bigl((\theta+\epsilon\omega)^n-\theta^n\bigr).
\]
Letting \(\epsilon\to 0^+\) gives
\[
\liminf_{k\to\infty}E_k(u)\ge E_\theta(u),
\]
which proves \eqref{lower bound on energy}.
\end{proof}

\begin{rmk}\label{pitfall for finite energy potentials}
The argument proving \eqref{lower bound on energy} fails for general \(u\in \E^1(X,\theta)\), since one cannot in general guarantee that the initial tangent \(\dot u_0\) of the weak geodesic \(t\mapsto u_t\) is integrable with respect to \((\theta+\epsilon\omega)^n\) for arbitrarily small \(\epsilon>0\). One obstruction comes from a counterexample of Di Nezza \cite[Example~4.5]{DN15}, which shows that for any \(\epsilon>0\),
\[
\E^1(X,\theta)\not\subset \E^1(X,\theta+\epsilon\omega).
\]
Indeed, let \(u\in \E^1(X,\theta)\setminus \E^1(X,\theta+\epsilon\omega)\) with \(u\le0\). Since \(\PSH(X,\theta)\subset \PSH(X,\theta+\epsilon\omega)\), the envelope construction shows that the weak geodesic \(v_t\) joining \(0\) and \(u\) in \(\PSH(X,\theta+\epsilon\omega)\) satisfies
\[
v_t \le u_t,\qquad t\in[0,1].
\]
Since \(u\notin \E^1(X,\theta+\epsilon\omega)\), one can show that
\[
\int_X \dot v_0\,(\theta+\epsilon\omega)^n = -\infty.
\]
As \(u_t\) and \(v_t\) agree at the endpoints, we have \(\dot u_0\le \dot v_0\le 0\). Therefore \(\dot{u}_0\notin L^1\bigl(X,(\theta+\epsilon\omega)^n\bigr)\), since
\[
\int_X \dot u_0\,(\theta+\epsilon\omega)^n
\le \int_X \dot v_0\,(\theta+\epsilon\omega)^n
= -\infty.
\]
Thus a necessary condition for the integrability of \(\dot u_0\) is that \(u\in \E^1(X,\theta+\epsilon\omega)\). To the author's knowledge, there is currently no general sufficient condition ensuring this.
\end{rmk}
\printbibliography
\noindent {\sc University of Maryland, College Park, USA}\\
{\tt yhou1994@umd.edu}\vspace{0.1in}\\

\end{document}